\documentclass[english,reqno]{amsart}

\usepackage{amssymb} 
\usepackage{enumitem} 
\usepackage{mathtools} 
\usepackage[table]{xcolor} 
\usepackage[all]{xy} 
\usepackage{tikz} 
\usepackage{indentfirst} 
\usepackage{babel} 
\usepackage{setspace} 

\usepackage[colorlinks,linkcolor=red,anchorcolor=green,citecolor=blue]{hyperref} 
\hypersetup{linktocpage = true} 

\usepackage{rotating} 

\usepackage{ytableau} 
\usepackage{longtable} 
\newcolumntype{M}[1]{>{\centering\arraybackslash}m{#1}} 

\usepackage{mathpazo}
\usepackage{mathrsfs}
\DeclareFontFamily{OMS}{rsfs}{\skewchar\font'60}
\DeclareFontShape{OMS}{rsfs}{m}{n}{<-5>rsfs5 <5-7>rsfs7 <7->rsfs10 }{}
\DeclareSymbolFont{rsfs}{OMS}{rsfs}{m}{n}
\DeclareSymbolFontAlphabet{\scr}{rsfs}
\DeclareSymbolFontAlphabet{\scr}{rsfs}

\usepackage[T1]{fontenc}

\newcommand\cE{{\mathcal E}}

\newcommand\cI{{\mathcal I}}
\newcommand\cK{{\mathcal K}}

\newcommand\cO{{\mathcal O}}

\newcommand\cV{{\mathcal V}}

\newcommand\bbC{{\mathbb C}}
\newcommand\bbF{{\mathbb F}}

\newcommand\bbP{{\mathbb P}}
\newcommand\bbQ{{\mathbb Q}}
\newcommand\bbR{{\mathbb R}}

\DeclareMathOperator*{\bs}{Bs}

\DeclareMathOperator*{\pic}{Pic}

\DeclareMathOperator*{\mult}{mult}
\DeclareMathOperator*{\conv}{conv}

\newcommand{\Chow}[1]{\ensuremath{\mbox{\rm Chow}(#1)}}

\newcommand{\RC}[1]{\ensuremath{\mbox{\rm RatCurves}^n(#1)}}

\newtheorem{lemma1}{}[section]
\newenvironment{lemma}{\begin{lemma1}{\bf Lemma.}}{\end{lemma1}}
\newenvironment{example}{\begin{lemma1}{\bf Example.}\rm}{\end{lemma1}}
\newenvironment{thm}{\begin{lemma1}{\bf Theorem.}}{\end{lemma1}}
\newenvironment{prop}{\begin{lemma1}{\bf Proposition.}}{\end{lemma1}}
\newenvironment{cor}{\begin{lemma1}{\bf Corollary.}}{\end{lemma1}}
\newenvironment{remark}{\begin{lemma1}{\bf Remark.}\rm}{\end{lemma1}}
\newenvironment{remarks}{\begin{lemma1}{\bf Remarks.}\rm}{\end{lemma1}}
\newenvironment{defn}{\begin{lemma1}{\bf Definition.}}{\end{lemma1}}
\newenvironment{defn-prop}{\begin{lemma1}{\bf Definition-Proposition.}}{\end{lemma1}}

\newenvironment{conjecture}{\begin{lemma1}{\bf Conjecture.}}{\end{lemma1}}
\newenvironment{question}{\begin{lemma1}{\bf Question.}}{\end{lemma1}}

\newenvironment{thm A}{{\bf Theorem A.}}{}
\newenvironment{thm B}{{\bf Theorem B.}}{}
\newenvironment{thm C}{{\bf Theorem C.}}{}
\newenvironment{thm D}{{\bf Theorem D.}}{}
\newenvironment{remark*}{{\bf Remark.}}{}
\newenvironment{example*}{{\bf Example.}}{}
\newenvironment{assumption*}{{\bf Assumption.}}{}

\newenvironment{conclusion*}{{\bf Conclusion.}}{}

\setlist[itemize]{leftmargin=*}
\setlist[enumerate]{leftmargin=*}

\numberwithin{equation}{section} 

\makeatletter

\ifnum\@ptsize=0  \fi \ifnum\@ptsize=2  \fi

\title{Seshadri constants of the anticanonical divisors of Fano manifolds with large index} 
\date{\today}


\author{Jie Liu}

\address{Jie Liu, Morningside Center of Mathematics, Academy of Mathematics and Systems Science, Chinese Academy of Sciences, Beijing, 100190, China}
\email{jliu@amss.ac.cn}

\begin{document}

\begin{abstract}
	Seshadri constants, introduced by Demailly, measure the local positivity of a nef divisor at a point. In this paper, we compute the Seshadri constants of the anticanonical divisors of Fano manifolds with coindex at most $3$ at a very general point. As a consequence, if $X$ is a nonsingular Fano threefold which is very general in its deformation family, then $\varepsilon(X,-K_X;x)\leq 1$ for all points $x\in X$ if and only if $\vert-K_X\vert$ is not base point free.
\end{abstract}

\maketitle

\tableofcontents

\vspace{-0.2cm}

\section{Introduction}

\subsection{Motivation}

Throughout the paper, we work over the field $\bbC$. Let $L$ be a nef line bundle over an $n$-dimensional projective normal variety $X$. In \cite{Demailly1992}, Demailly introduces an interesting invariant which measures the local positivity of $L$ at a point $x\in X$.

\begin{defn}\label{Definition Seshadri constant}
	Let $X$ be a projective normal variety and let $L$ be a nef line bundle on $X$. To every smooth point $x\in X$, we attach the number
	\[\varepsilon(X,L;x)\colon=\inf_{x\in C}\frac{L\cdot C}{\mult_x C},\]
	which is called the Seshadri constant of $L$ at $x$. Here the infimum is taken over all irreducible curves $C$ passing through $x$ and $\mult_x C$ is the multiplicity of $C$ at $x$.
\end{defn}

The Seshadri constant is a lower-continuous function over $X$ in the topology which closed sets are countable unions of Zariski closed sets. Moreover, there is a number, which we denote by $\varepsilon(X,L;1)$, such that it is the maximal value of Seshadri constant on $X$. This maximum is attained for a very general point $x\in X$. Unfortunately, in general it is very difficult to compute or estimate $\varepsilon(X,L;1)$. For the upper bound, an elementary observation shows that $\varepsilon(X,L;1)\leq \sqrt[n]{L^n}$. There have been many works in trying to give lower bound for this invariant. Ein and Lazarsfeld show that if $X$ is surface and $L$ is ample, then $\varepsilon(X,L;1)\geq 1$ (see \cite[Theorem]{EinLazarsfeld1993}). In higher dimension, Ein, K{\"u}chle and Lazarsfeld prove that $\varepsilon(X,L;1)\geq 1/\dim(X)$ for any ample line bundle $L$ over $X$ (see \cite[Theorem 1]{EinKuechleLazarsfeld1995}) and this bound has been improved by Nakamaye in \cite{Nakamaye2005}, by Cascini-Nakamaye for $3$-folds in \cite{CasciniNakamaye2014} and by Lee for Fano manifolds in \cite{Lee2003}. In general, we have the following conjecture.

\begin{conjecture}\cite[Conjecture 5.2.4]{Lazarsfeld2004}\label{Intro ConjecturelowerboundSeshadri}
	Let $L$ be an ample line bundle over a projective manifold $X$. Then $\varepsilon(X,L;1)\geq 1$.
\end{conjecture}

In this paper, we are interested in the case where $X$ is a Fano manifold and $L$ is the anticanonical divisor $-K_X$. Given an $n$-dimensional Fano manifold $X$, Bauer and Szemberg show that $\varepsilon(X,-K_X;1)\leq n+1$ with equality if and only if $X$ is isomorphic to the projective space $\bbP^n$, and recently this result is generalized to $\bbQ$-Fano varieties by Liu-Zhuang in \cite{LiuZhuang2018}. On the other hand, as predicted by Conjecture $\ref{Intro ConjecturelowerboundSeshadri}$, we should have $\varepsilon(X,-K_X;1)\geq 1$ (see \cite{Lee2003} for more related results). Thus it is natural and interesting to consider the following question. 

\begin{question}\label{1-Classification}
	Classify all Fano manifolds $X$ with $\varepsilon(X,-K_X;1)\leq 1$.
\end{question}

In dimension $2$, the Seshadri constants of the anticanonical divisors of del Pezzo surfaces are computed by Broustet in \cite{Broustet2006}. 

\begin{thm}\cite[Th\'eor\`eme 1.3]{Broustet2006}\label{Broustet's theorem}
	Let $S$ be a del Pezzo surface of degree $d$.
	\begin{enumerate}
		\item $\varepsilon (S,-K_S;1)=1$ if $d=1$.
		
		\item $\varepsilon(S,-K_S;1)=4/3$ if $d=2$.
		
		\item $\varepsilon(S,-K_S;1)=3/2$ if $d=3$.
		
		\item $\varepsilon(S,-K_S;1)=2$ if $4\leq d\leq 8$.
		
		\item $\varepsilon(S,-K_S;1)=3$ if $d=9$.
	\end{enumerate}
\end{thm}

This gives a complete answer to Question \ref{1-Classification} for del Pezzo surfaces. As a direct corollary, if $S$ is a del Pezzo surface, then $\varepsilon(S,-K_S;1)\leq 1$ if and only if $\vert-K_S\vert$ is not base point free. 

\subsection{Main Results}

Recall that the \emph{index} $r_X$ of an $n$-dimensional Fano manifold $X$ is defined to be the largest positive integer dividing $-K_X$ in $\pic(X)$ and the divisor $H$ such that $-K_X\sim r_XH$ is called the \emph{fundamental divisor} of $X$. The \emph{coindex} of $X$ is defined to be $n+1-r_X$. As predicted by Conjecture \ref{Intro ConjecturelowerboundSeshadri}, we should have
\begin{equation}\label{Inequality-Anticanonical}
	\varepsilon(X,-K_X;1)=r_X\varepsilon(X,H;1)\geq r_X.
\end{equation}
The inequality \eqref{Inequality-Anticanonical} is confirmed in the cases $r_X\geq n-2$ or $n\leq 4$ by Broustet in \cite{Broustet2009}. Following the same idea, we can easily derive the following result by applying the main result of \cite{Liu2017} and the results proved in \cite{Broustet2009}.

\begin{thm}\label{Intro Lowerbound}
	Let $X$ be a $n$-dimensional Fano manifold with index $r_X$ at least $ n-3$, then $\varepsilon(X,-K_X;1)\geq r_X$.
\end{thm}

\begin{proof}
	Denote by $H$ the fundamental divisor of $X$. By \cite[Th{\'e}or{\`e}me 1.5]{Broustet2009}, it remains to consider the case $r_X=n-3$. Thanks to \cite[Theorem 1.2]{Liu2017} and \cite[Theorem 1.1]{Floris2013}, there is a descending sequence of subvarieties
	\[X=X_n\supsetneq X_{n-1}\supsetneq \cdots\supsetneq X_{3}\]
	such that $X_{i}\in \vert H\vert_{X_{i+1}}\vert$ has at worst Gorenstein canonical singularities. Moreover, according to \cite[Th\'eor\`eme 1.4 (2)]{Broustet2009}, we have $\varepsilon(X_{3},H\vert_{X_{3}};1)\geq 1$ as $K_{X_3}\sim 0$. Then one can apply verbatim the proof of \cite[Th\'eor\`eme 1.5]{Broustet2009} to prove the theorem.
\end{proof}

The main result of this paper is to compute $\varepsilon(X,-K_X;1)$ explicitly for Fano manifolds $X$ with coindex at most $3$. Before giving the precise statement, we recall the \emph{minimal anticanonical degree} $\ell_X$ of a covering family of minimal rational curves on $X$ by $\ell_X$, so that $\ell_X\in\{2,\cdots,n+1\}$ and $\varepsilon(X,-K_X;1)\leq \ell_X$. 

\begin{thm}\label{Intro Dimension at least four Seshadri}
	Let $X$ be an $n$-dimensional Fano manifold with index $r_X\geq \max\{2,n-2\}$. Then through every point $x\in X$ there is a rational curve $C$ such that $-K_X\cdot C=r_X$. In particular, we have $\varepsilon(X,-K_X;1)=\ell_X=r_X$.
\end{thm}

Now it remains to consider nonsingular Fano threefolds. The classification of polarized Fano threefolds $(X,L)$ with $\varepsilon(X,L;x)<1$ for some point $x\in X$ was studied by Lee in \cite{Lee2003,Lee2004a}. Moreover, if $X$ is a very general nonsingular Fano threefold with $\rho(X)=1$, then $\varepsilon(X,-K_X;1)$ is calculated by Ito via toric degenerations (see \cite[Theorem 1.8]{Ito2014}). The nonsingular Fano threefolds with $\rho\geq 2$ are classified by Mori-Mukai in \cite{MoriMukai1981/82,MoriMukai2003}. Given a nonsingular Fano threefold $X$, we identify it (or rather its deformation family) by the pair of numbers
\[\gimel(X)=\rho.N,\]
where $\rho$ is the Picard rank of the threefold $X$, and $N$ is its number in the classification tables in \cite{MoriMukai1981/82,Shafarevich1999,MoriMukai2003}. 

\begin{thm}\label{Intro Fano threefolds Picard number large Seshadri}
	Let $X$ be a nonsingular Fano threefold with $\rho(X)\geq 2$. 
	\begin{enumerate}
		\item $\varepsilon(X,-K_X;1)=1$ if and only if $X$ admits a del Pezzo fibration of degree $1$, or equivalently $\gimel(X)\in \{2.1, 10.1\}$.
		
		\item $\varepsilon(X,-K_X;1)=4/3$ if and only if $X$ admits a del Pezzo fibration of degree $2$, or equivalently $\gimel(X)\in\{2.2, 2.3, 9.1\}$.
		
		\item $\varepsilon(X,-K_X;1)=3/2$ if and only if $X$ admits a del Pezzo fibration of degree $3$, or equivalently $\gimel(X)\in\{2.4, 2.5, 3.2, 8.1\}$
		
		\item $\varepsilon(X,-K_X;1)=3$ if and only if $X$ is isomorphic to the blow-up of $\bbP^3$ along a smooth plane curve $C$ of degree $d\leq 3$, or equivalently $\gimel(X)\in\{2.28, 2.30, 2.33\}$.
		
		\item $\varepsilon(X,-K_X;1)=2$ otherwise.
	\end{enumerate}
\end{thm}

Combined with \cite[Theorem 1.8]{Ito2014} and \cite[Theorem 2.4.5]{Shafarevich1999} (cf. Corollary \ref{1-Non-Basepoint-Freeness}), Theorem \ref{Intro Fano threefolds Picard number large Seshadri} provides an answer to Question \ref{1-Classification} for nonsingular Fano threefolds modulo genericity assumptions.

\begin{cor}\label{Intro Seshadri=1 classification}
	Let $X$ be a nonsingular Fano threefold which is very general in its deformation family. Then $\varepsilon(X,-K_X;1)\leq 1$ if and only if $\vert -K_X\vert$ is not base point free.
\end{cor}

\subsection{Further developments}

One may ask if Corollary \ref{Intro Seshadri=1 classification} still holds without the genericity assumption. In fact, the genericity assumption is only used to apply the result of Ito. In the case where $-K_X$ is very ample, by \cite[Lemma 2.2]{Chan2010}, $\varepsilon(X,-K_X;1)\leq 1$ if and only if $(X,-K_X)$ is covered by lines, which never happens if $r_X=1$. If $X$ is a nonsingular Fano threefold with $\rho(X)=1$, then $-K_X$ is very ample except $\gimel(X)\in \{1.1,1.2\}$ (see \cite[Proposition 4.1.11]{Shafarevich1999}). If $\gimel(X)\in\{1.1, 1.2\}$ and $-K_X$ is not very ample, then $X$ is actually a smooth complete intersection in a weighted projective space (cf. Proposition \ref{Threefolds-Picard-Number-One} and Remark \ref{Remark-Weighted-Complete-Intersection}). We are thus led to ask the following question.

\begin{question}
	Let $X$ be a smooth Fano weighted complete intersection in a weighted projective space, and let $\cO_X(1)$ be the restriction of the universal $\cO(1)$-sheaf from the weighted projective space. If $(X,\cO_X(1))$ is not covered by lines, does there exist a point $x\in X$ such that $\varepsilon(X,\cO_{X}(1);x)>1$?
\end{question}

Furthermore, another natural question is to ask if the analogue of Corollary \ref{Intro Seshadri=1 classification} holds in higher dimension. By Theorem \ref{Intro Lowerbound}, if $X$ is an $n$-dimensional Fano manifold of coindex at most $4$, an obvious necessary condition for $\varepsilon(X,-K_X;1)\leq 1$ is that the index $r_X$ of $X$ must equal to $1$. On the other hand, whereas the non-freeness of $\vert-K_X\vert$ implies $r_X=1$ for nonsingular Fano $4$-fold, it is no longer true for nonsingular Fano $5$-folds. Even in dimension $4$, there exist nonsingular Fano $4$-folds $X$ such that $\vert-K_X\vert$ is not base point free but $\varepsilon(X,-K_X;1)>1$.

\begin{example}\label{4-fold-5-fold}
	\begin{enumerate}
		\item Let $X\subset \bbP(1^{4},2,5)$ be a very general hypersurface of degree $10$. Then $X$ is a nonsingular Fano $4$-fold with index $1$. By \cite[Theorem 2.2]{IltenLewisPrzyjalkowski2013}, $(X,\cO_X(1))$ degenerates to a polarized toric variety $(X_P,L_P)$ as a $\bbQ$-polarized variety for 
		\[P\colon=\conv(e_1,e_2,e_3,e_4,-1/6(e_1+e_2+e_3+e_4))\subset\bbR^4.\] By \cite[Example 3.10 and Lemma 4.3]{Ito2014}, we obtain
		\[\varepsilon(X,\cO_X(1);1)\geq \varepsilon(X_P,L_P;1)\geq 10/9.\]
		
		\item Similarly, let $X\subset\bbP(1^5,2,5)$ be a nonsingular hypersurface of degree $10$. Then $X$ is a nonsingular Fano $5$-fold with index $2$ and $\vert-K_X\vert$ is not base point free. By \cite[V, 4.11]{Kollar1996}, the polarized pair $(X,\cO_X(1))$ is covered by lines. In particular, we have $\varepsilon(X,-K_X;1)=2$.
	\end{enumerate}
\end{example}

Thus the strict analogue of Corollary \ref{Intro Seshadri=1 classification} does not hold in higher dimension. However we can still ask the following question.

\begin{question}
	Let $X$ be a Fano manifold such that $\varepsilon(X,-K_X;1)\leq 1$. Is the base locus of the linear system $\vert-K_X\vert$ non-empty?
\end{question}

\subsection*{Acknowledgements} The author wishes to express his thanks to Andreas H\"oring and Christophe Mourougane their constant encouragements and supports. He also wound like to thank Ama\"el Broustet and St\'ephane Druel for useful comments. This paper was written while the author stayed at Institut de Recherche Math\'ematique de Rennes (IRMAR) and Laboratoire de Math\'ematiques J.\,A. Dieudonn\'e (LJAD) and he would like to thank both the institutions for the hospitality and support.

\section{Polarized manifolds covered by lines}\label{Section Lines Fano manifolds}

In this section, we study the existence of lines on Fano manifolds with large index and the main aim is to prove Theorem \ref{Intro Dimension at least four Seshadri}.

\begin{defn}\label{Lines}
	Let $(X,H)$ be a polarized projective manifold. A line (with respect to $H$) in $X$ is a rational curve $C\subset X$ such that $H\cdot C=1$. We say that $(X,H)$ is covered by lines if through every point $x$ of $X$ there is a line contained in $X$.  
\end{defn}

In general $X$ cannot be embedded into projective spaces in such a way that a line $C$ on $X$ becomes a projective line. If $(X,H)$ is covered by lines, then by definition it is easy to see $\varepsilon(X,H;x)\leq 1$ for every point $x\in X$.

\begin{lemma}\label{RCfamily}
	Let $(X,H)$ be a polarized projective manifold. Assume moreover that through a very general point $x$ there is a rational curve $C_x\subset X$ of degree $d$. Then there exists an irreducible closed subvariety $W$ of $\Chow{X}$ such that 
	\begin{enumerate}
		\item the universal cycle over $W$ dominates $X$, and
		
		\item the subset of points in $W$ parametrizing the rational curves $C_x$ (viewed as $1$-cycles on $X$) is dense in $W$. 
	\end{enumerate}
\end{lemma}

\begin{proof}
	Recall that $\Chow{X}$ has countably many irreducible
	components. On the other hand, since we are working over $\bbC$, we have uncountably many lines on $X$. Then the existence of $W$ is clear.
\end{proof}

\begin{remark}\label{covered by lines open property}
	Let $(X,H)$ be a polarized projective manifold. If through a very general point $x\in X$ there is a line $x\in\ell \subset X$, then through every point $x\in X$ there is a line. In fact, let us denote by $W$ the subvariety of $\Chow{X}$ provided in Lemma \ref{RCfamily}. We remark that every cycle $[C]$ in $W$ is irreducible and reduced as $H\cdot C=1$. Let $\cV$ be an irreducible component of $\RC{X}$ such that its image in $\Chow{X}$ contains $W$. Then $\cV$ is an unsplit covering family of minimal rational curves. Let $U$ be the universal family over $\cV$. Then the evaluation map $e\colon U\rightarrow X$ is surjective. Hence, $(X,H)$ is covered by lines. 
\end{remark}

The following fact is well-known for experts, but for the convenience of reader we give a complete proof.

\begin{prop}\label{Lines-Large-Index}
	Let $X$ be an $n$-dimensional Fano manifold such that $n\geq 3$, and let $H$ be the fundamental divisor. If either $r_X>n/2$, or $r_X=n/2$ and $\rho(X)\geq 2$, then $(X,H)$ is covered by lines.
\end{prop}

\begin{proof}
	Let $\cK$ be a family of minimal rational curves on $X$. By \cite[V, Theorem 1.6]{Kollar1996}, for a general member $[C]\in \cK$, we have $-K_X\cdot C\leq n+1$. Moreover, according to \cite{ChoMiyaokaShepherd-Barron2002}, the equality holds if and only if $X$ is isomorphic to $\bbP^n$. Thus, without loss of generality, we may assume that $-K_X\cdot C\leq n$ for a general member $[C]\in \cK$. As $r_X\geq n/2$, for a general member $[C]\in\cK$, we obtain
	\[H\cdot C=\frac{1}{r_X}(-K_X\cdot C)\leq \frac{n}{r_X}\leq 2\]
	with equality if and only if $r_X=n/2$ and $-K_X\cdot C=n$. Let $X$ be an $n$-dimensional Fano manifold with index $n/2$ and $\rho(X)\geq 2$ such that $-K_X\cdot \cK=n$ for all families $\cK$ of minimal rational curves on $X$. Then we get $\ell_X=n$. Thanks to \cite[Theorem 1.4]{CasagrandeDruel2015}, $X$ is the blow-up of $\bbP^n$ along a smooth subvariety $A$ of dimension $n-2$ and degree $d\in\{1,\dots,n\}$, contained in a hyperplane. Nevertheless, it is easy to see that such $X$ are Fano manifolds of index $1$, which is impossible by our assumption. Hence, if $X$ is an $n$-dimensional Fano manifold with index $n/2$ and $\rho(X)\geq 2$, then we have $\ell_X=n/2$. In particular, $(X,H)$ is covered by lines.
\end{proof}

\begin{remark}
	If $X$ is an $n$-dimensional Fano manifold with $\rho(X)=1$ such that $\ell_X=n$ and $n\geq 3$, Miyaoka proves in \cite{Miyaoka2004} that $X$ is isomorphic to a smooth quadric hypersurface. However, the proof there is incomplete (see \cite[Remark 5.2]{DedieuHoering2017}).
\end{remark}

\begin{cor}\label{Lines-Del-Pezzo-Mukai}
	Let $X$ be an $n$-dimensional Fano manifold such that $n\geq 3$, and let $H$ be the fundamental divisor. If $r_X\geq \max\{n-2,2\}$, then $(X,H)$ is covered by lines.
\end{cor}

\begin{proof}
	By Proposition \ref{Lines-Large-Index}, it remains to consider the case $r_X=2$, $n=4$ and $\rho(X)=1$; that is, $X$ is a $4$-dimensional Mukai manifold with $\rho(X)=1$. Then $X$ is a smooth complete intersection in either a weighted projective space or a rational homogeneous space (see \cite[\S\,5.2]{Shafarevich1999}). In the former case, $(X,H)$ is covered by lines by \cite[V,4.11]{Kollar1996}. In the latter case, it is easy to check that $(X,H)$ is also covered by lines (see \cite[Lemma 1]{ItoMiura2014}).	
\end{proof}

\begin{proof}[Proof of Theorem \ref{Intro Dimension at least four Seshadri}]
	 It follows directly from Corollary \ref{Lines-Del-Pezzo-Mukai} and Theorem \ref{Intro Lowerbound}.
\end{proof}

\section{Characterize Fano threefolds via Seshadri constants}\label{Fano threefolds with Picard number at least two Section}

\subsection{Fano threefolds with Picard number one} 

As mentioned in the introduction, the Seshadri constants $\varepsilon(X,-K_X;1)$ of very general nonsingular Fano threefolds $X$ with $\rho(X)=1$ are computed by Ito in \cite{Ito2014}. In the following result, we show that $\varepsilon(X,-K_X;1)$ is invariant in its deformation family if $-K_X$ is very ample. Recall that the \emph{genus} of a nonsingular Fano threefold $X$ is defined to be $(-K_X)^3/2+1$.

\begin{prop}\label{Threefolds-Picard-Number-One}
	Let $X$ be a nonsingular Fano threefold of genus $g$ with $\rho(X)=1$. If $-K_X$ is very ample, then the Seshardri constant $\varepsilon(X,-K_X;1)$ is invariant in its deformation family.
\end{prop}

\begin{proof}
	If the index of $X$ is at least $2$, it follows directly from Theorem \ref{Intro Dimension at least four Seshadri}. Now we consider the nonsingular Fano threefolds with index one. By the very ampleness of $-K_X$, the induced morphism $\Phi_{\vert-K_X\vert}\colon X\rightarrow \bbP^{g+1}$ is an embedding, where $g$ is the genus of $X$. Note that $(X,-K_X)$ is not covered by lines since the index of $X$ is $1$. After identifying $X$ with its image under $\Phi_{\vert-K_X\vert}$, the line bundle $\cO_X(-K_X)$ is isomorphic to $\cO_{\bbP^{g+1}}(1)\vert_X$. By the classification (see \cite[Table 12.2]{Shafarevich1999}), except the case $g=4$, $X$ is always a complete intersection of hypersurfaces of degrees at most two  in a rational homogeneous space. Let $x\in X$ be a very general point. If $(X,-K_X)$ is not covered by lines, by the construction, there are no lines in $\bbP^{g+1}$ lying in $X$ and passing through $x$. Thus, thanks to \cite[Theorem 3]{ItoMiura2014}, for a very general point $x\in X$, we have $\varepsilon(X,-K_X;x)=\varepsilon(X,\cO_{\bbP^{g+1}}(1)\vert_X;x)=2$. For the case $g=4$, $X$ is a complete intersection of a quadric and a cubic (see \cite[Table 12.2]{Shafarevich1999}), by \cite[Theorem 3]{ItoMiura2014} again, we get $\varepsilon(X,-K_X;1)=\varepsilon(X,\cO_{\bbP^5}(1)\vert_X;1)=3/2$.
\end{proof}

\begin{remark}\label{Remark-Weighted-Complete-Intersection}
	According to \cite[Proposition 4.1.11]{Shafarevich1999}, if $X$ is a nonsingular Fano threefold with $\rho(X)=1$ such that $-K_X$ is not very ample, then either $X$ is a weighted hypersurface of degree $6$ in $\bbP(1^4,3)$ (i.e., $\gimel(X)=1.1$), or $X$ is a complete intersection of two weighted quadric hypersurfaces in $\bbP(1^5,2)$ (i.e., $\gimel(X)=1.2$). Up to now I do not know if the Seshadri constant $\varepsilon(X,-K_X;1)$ is invariant in the deformation families of smooth Fano complete intersections in weighted projective spaces.
\end{remark}

\subsection{Splitting and free splitting}

The following concept plays a key role in the classification of Fano threefolds, and it is also the main ingredient of the proof of Theorem \ref{Intro Fano threefolds Picard number large Seshadri}.

\begin{defn}\label{Splitting}
	A Weil divisor $D$ on a projective manifold $X$ has a splitting if there are two non-zero effective divisors $D_1$ and $D_2$ such that $D_1+D_2\in\vert D\vert$. The splitting is called free if the linear systems $\vert D_1\vert$ and $\vert D_2\vert$ are base point free.
\end{defn}

The following criterion is frequently used in \cite{MoriMukai1986} to check the free splitting of anticanonical divisors.

\begin{defn-prop}\cite[Proposition 2.10]{MoriMukai1986}\label{intersection of members}
	Let $Y$ be a projective manifold. Assume that $C$ is a smooth proper closed subscheme of $Y$. Let $\cI_C$ be the sheaf of ideals of $C$ in $X$, and let $D$ be a divisor on $Y$. Let $f\colon X\rightarrow Y$ be the blow-up of $Y$ along $C$. We denote by $E$ the exceptional divisor of $f$. We say that $C$ is an intersection of members of $\vert D\vert$ when the equivalent conditions below are satisfied.
	\begin{enumerate}
		\item The map $H^0(Y,\cO_Y(D)\otimes\cI_C)\otimes\cO_Y\rightarrow \cO_Y(D)\otimes\cI_C$ is surjective.
		
		\item The linear system $\vert f^*D-E\vert$ is base point free.
	\end{enumerate}
\end{defn-prop}

\begin{lemma}\label{Criterion-Complete-Intersection}
	Under the situation of Definition-Proposition \ref{intersection of members}, if $C$ is a curve, then the linear system $\vert f^*D-E\vert$ is composed with a pencil of surfaces if and only if $C$ is a complete intersection of members of $\vert D\vert$.
\end{lemma}

\begin{proof}
	One implication is clear. Now we assume that $\vert f^*D-E\vert$ is composed with a pencil of surfaces. Let $A$ be an ample divisor over $Y$. Then the pull-back $f^*A$ is nef and big. Since $\vert f^*D-E\vert$ is composed with a pencil of surfaces, the numerical dimension of $f^*D-E$ is $1$. In particular, we have $(f^*D-E)^2\cdot f^*A=0$. Then it yields
	\[D^2\cdot A=(f^*D)^2\cdot f^*A=-E^2\cdot f^*A=A\cdot C.\]
	Since $H^0(Y,\cO_Y(D)\otimes\cI_C)\otimes\cO_Y\rightarrow \cO_Y(D)\otimes\cI_C$ is surjective, there exist $D_1, D_2\in \vert D\vert$ without common components such that $C\subset D_1\cap D_2$ as sets. On the other hand, we have
	\[A\cdot C=D^2\cdot A=D_1\cdot D_2\cdot A.\]
	As $A$ is ample, we obtain $C=D_1\cap D_2$ as $1$-cycles. Since $C$ is smooth, we get $C=D_1\cap D_2$ as scheme-theoretical complete intersections.
\end{proof}

The following theorem plays a key role in the proof of Theorem \ref{Intro Fano threefolds Picard number large Seshadri}. It is claimed in \cite{MoriMukai1981/82} and the proof is provided in \cite{MoriMukai1986}.

\begin{thm}\cite[Theorem 3]{MoriMukai1986}\label{existence of free splitting}
	Let $X$ be a nonsingular Fano threefold with $\rho(X)\geq 2$, then the anticanonical divisor $-K_X$ has a splitting. Furthermore, $-K_X$ has a free splitting if $\vert-K_X\vert$ is base point free.
\end{thm}

\subsection{Morphisms induced by splittings} 

This subsection is devoted to study the relation between $\varepsilon(X,-K_X;1)$ and the maps induced by splittings of $-K_X$. We begin with a simple but useful observation.

\begin{lemma}\label{Contraction-Curve}
	Let $X$ be an $n$-dimensional Fano manifold, and let $g\colon X\rightarrow Y$ be a surjective morphism with connected fibers onto a normal projective variety $Y$. Let $x\in X$ be a very general point, and let $C$ be an irreducible curve passing through $x$ and contracted by $g$. If $\dim(Y)=n-1$, then $-K_X\cdot C=2\mult_x C$.
\end{lemma}

\begin{proof}
	By genericity assumption, we may assume that the fiber of $g$ passing through $x$ is irreducible and smooth. In particular, as $\dim(Y)=n-1$, $C$ is exactly the fiber of $f$ over $x$. As $-K_X$ is ample, $C$ is a rational curve. Hence, we have $-K_X\cdot C=2=2\mult_x C$ by the smoothness of $C$. 
\end{proof}

Now we can describe the structure of nonsingular Fano threefolds with small $\varepsilon(X,-K_X;1)$.

\begin{thm}\label{Morphisms-Splittings}
	Let $X$ be a nonsingular Fano threefold with $\rho(X)\geq 2$. If $\vert-K_X\vert$ is base point free, then $\varepsilon(X,-K_X;1)>1$. Moreover, if $\varepsilon(X,-K_X;1)<2$, for any free splitting $-K_X=D_1+D_2$, after exchanging $D_1$ and $D_2$ if necessary, one of the following holds.
	\begin{enumerate}
		\item $\varepsilon(X,-K_X;1)=4/3$, and $\vert D_1\vert$ induces a del Pezzo fibration $X\rightarrow \bbP^1$ of degree $2$.
		
		\item $\varepsilon(X,-K_X;1)=3/2$, and $\vert D_1\vert$ induces a del Pezzo fibration $X\rightarrow \bbP^1$ of degree $3$.
	\end{enumerate}
\end{thm}

\begin{proof}
	Fix a free splitting $-K_X=D_1+D_2$ (cf. Theorem \ref{existence of free splitting}). Let $g_i\colon X\rightarrow Y_i$ be the morphism induced by the free linear system $\vert D_i\vert$. Moreover, by Stein factorization and generic smoothness, we shall assume that the general fiber of $g_i$ is irreducible and smooth. Let $x\in X$ be a very general point, and let $C$ be an irreducible curve passing through $x$. There are three different possibilities for $C$.
	
	\begin{enumerate}
		\item The curve $C$ is not contracted by $g_1$ nor $g_2$.
		
		\item The curve $C$ is contracted by $g_1$ (resp. $g_2$) and $g_1$ (resp. $g_2$) is a fibration in curves.
		
		\item The curve $C$ is contracted by $g_1$ (resp. $g_2$) and $g_1$ (resp. $g_2$) is a fibration in surfaces.
	\end{enumerate}
	
	In case (1), by the freeness of $\vert D_1\vert$ and $\vert D_2\vert$, we can find $\widetilde{D}_1\in\vert D_1\vert$ and $\widetilde{D}_2\in\vert D_2\vert$ both passing through $x$ and not containing $C$. In particular, we have 
	\begin{equation}\label{Non-Contraction}
	-K_X\cdot C=(\widetilde{D}_1+\widetilde{D}_2)\cdot C\geq 2 {\rm{mult}_x C}.
	\end{equation}
	
	In case (2), without loss of generality, we may assume that $C$ is contracted by $g_1$. By Lemma \ref{Contraction-Curve}, we have 
	\begin{equation}\label{Contraction-To-Curve}
	-K_X\cdot C=2\rm{mult}_x C.
	\end{equation}
	
	In case (3), without loss of genericity, we shall assume that $C$ is contracted by $g_1$. Let $S$ be the fiber of $g_1$ passing through $x$. Since $x$ is very general, we can assume that $S$ is a general fiber of $g_1$. As $-K_X$ is ample, $S$ is a smooth del Pezzo surface and we have 
	\begin{equation}\label{Contraction-To-Surface}
	-K_X\cdot C=-K_S\cdot C \geq \varepsilon(S,-K_S;1)\rm{mult}_x C.
	\end{equation}
	
	As a consequence, $\varepsilon(X,-K_X;1)<2$ only if one of $g_1$ and $g_2$ is a fibration in del Pezzo surfaces of degree at most $3$ (cf. Theorem \ref{Broustet's theorem}). Moreover, if $\varepsilon(X,-K_X;1)=1$, then one of $g_1$ and $g_2$ is a fibration in del Pezzo surfaces of degree $1$. However, the anticanonical linear system $\vert-K_S\vert$ of a del Pezzo surface $S$ of degree $1$ is not base point free, we get a contradiction. Thus we have always $\varepsilon(X,-K_X;1)>1$. 
	
	On the other hand, note that $-K_X$ is ample and $-K_X=D_1+D_2$, so there are no curves contracted by both $g_1$ and $g_2$. In particular, at most one of $g_1$ and $g_2$ is a fibration in surfaces. Combined with \eqref{Non-Contraction}, \eqref{Contraction-To-Curve} and \eqref{Contraction-To-Surface}, Theorem \ref{Broustet's theorem} implies that if $\varepsilon(X,-K_X;1)<2$, then either $\varepsilon(X,-K_X;1)=4/3$ and one of $g_1$ and $g_2$ is a fibration of del Pezzo surfaces of degree $2$, or $\varepsilon(X,-K_X;1)=3/2$ and one of $g_1$ and $g_2$ is a fibration of del Pezzo surfaces of degree $3$.
\end{proof}

In the case where $\vert-K_X\vert$ is not base point free, we have the following well-known classification result.

\begin{thm}\cite[Theorem 2.4.5]{Shafarevich1999}\label{Non-Basepoint-Free}
	Let $X$ be a nonsingular Fano threefold. Then the linear system $\vert-K_X\vert$ is base point free except for the following two cases.
	\begin{enumerate}
		\item The blow-up of $V_1$ along an elliptic curve which is an intersection of two divisors from $\vert -\frac{1}{2}K_{V_1}\vert$, where $V_1$ is a smooth del Pezzo threefold of degree $1$, i.e. $\gimel(X)=2.1$.
		
		\item $\bbP^1\times S_1$, where $S_1$ is a smooth del Pezzo surface of degree $1$, i.e. $\gimel(X)=10.1$.
	\end{enumerate}
\end{thm}

\begin{cor}\label{1-Non-Basepoint-Freeness}
	Let $X$ be a nonsingular Fano threefold with $\rho(X)\geq 2$. Then the following three statements are equivalent.
	\begin{enumerate}
		\item The linear system $\vert-K_X\vert$ is not base point free.
		
		\item There is a del Pezzo fibration $X\rightarrow \bbP^1$ of degree $1$.
		
		\item $\varepsilon(X,-K_X;1)=1$.
	\end{enumerate}
\end{cor}

\begin{proof}
	$(1) \implies (2)$. It is enough to consider case (1) of Theorem \ref{Non-Basepoint-Free}. Let $E$ be the exceptional divisor of the blow-up $\pi\colon X\rightarrow V_1$. Let $H$ be the fundamental divisor of $V_1$. Then $C$ is an intersection of members from $\vert H\vert$. Thus the linear system $\vert \pi^*H-E\vert$ is base point free. Moreover, since $C$ is a smooth complete intersection, $\vert \pi^*H-E\vert$ is composed with a pencil of del Pezzo surfaces of degree $H^3=1$.
	
	$(2) \implies (3)$. Let $x\in X$ be a very general point. Thanks to \cite[Th\'eor\`eme 1.4]{Broustet2009}, it suffices to show $\varepsilon(X,-K_X;x)\leq 1$. Let $S$ be the fiber of the fibration $X\rightarrow \bbP^1$ containing $x$. By the genericity of $x$ and generic smoothness, $S$ is a smooth del Pezzo surface of degree $1$. By definition, we get
	\[\varepsilon(X,-K_X;x)\leq \varepsilon(S,-K_X\vert_S;x)=\varepsilon(S,-K_S;x)\leq \varepsilon(S,-K_S;1)=1.\]
	The last inequality follows from the lower semi-continuity of Seshadri constant.
	
	$(3) \implies (1)$. It follows directly from Theorem \ref{Morphisms-Splittings}.
\end{proof}

\begin{remark}\label{Splitting-Nonfreeness}
	If $X$ is a nonsingular Fano threefold with $\rho(X)\geq 2$ such that $\vert-K_X\vert$ is not base point free, then there exists a splitting $-K_X=D_1+D_2$ such that $\vert D_1\vert$ is base point free and $D_2$ is nef and big. In fact, we can set $D_1=\pi^*H-E$ for $\gimel(X)=2.1$ and $D_1=p^*\cO_{\bbP^1}(1)$. Then the nefness and bigness of $D_2$ can be easily checked. 
\end{remark}

\subsection{A structure theorem} 

The following example shows that a smooth Fano threefold may be released as two del Pezzo fibrations of different degree. 

\begin{example}
	Let $X=S_d\times\bbP^1$, where $S_d$ is a smooth del Pezzo surface of degree $d$ such that $S_d$ is isomorphic to neither $\bbP^2$ nor a smooth quadric surface $Q^2$. Then there is natural fibration $S_d\rightarrow \bbP^1$ with general fiber $\bbP^1$. Denote the induced fibration $X\rightarrow\bbP^1$ by $p_1$ and the second projection $S_d\times\bbP^1\rightarrow\bbP^1$ by $p_2$. Then the general fiber of $p_1$ is a smooth quadric surface, while the fiber of $p_2$ is a del Pezzo surface of degree $d$.
\end{example}

\begin{lemma}\label{Fibration-Different}
	Let $X$ be a nonsingular Fano threefold such that $\bs\vert -K_X\vert\not=\emptyset$. If $X$ admits a del Pezzo fibration of degree $d\geq 2$, then $d\geq 4$.
\end{lemma}

\begin{proof}
	As explained in Remark \ref{Splitting-Nonfreeness}, there exists a splitting $-K_X=D_1+D_2$ such that $\vert D_1\vert$ is base point free and $D_2$ is nef and big. Moreover, it is easy to see that $\vert D_1\vert$ is composed with a pencil of del Pezzo surfaces of degree $1$. Denote by $g\colon X\rightarrow \bbP^1$ a del Pezzo fibration of degree $d\geq 2$ and let $S$ be a general fiber of $g$. Then $-K_S=D_1\vert_S+D_2\vert_S$ is a splitting. As $\vert D_1\vert$ is base point free, $\vert D_1\vert_S\vert$ is also base point free. Denote by $\pi\colon X\rightarrow Y$ the induced surjective morphism with connected fibers. Then $Y$ is a curve. Let $s\in S$ be a very general point and let $C\subset S$ be an irreducible curve passing through $s$. If $C$ is not contracted by $\pi$, then we can find $D\in \vert D_1\vert_S\vert$ passing through $s$ but not containing $C$. Since $D_2\vert_S$ is a nef and big divisor, by \cite[Proposition 4.12]{Broustet2009}, we get
	\[-K_S\cdot C\geq D\cdot C+\varepsilon(S,D_2\vert_C;1)\textrm{mult}_xC\geq 2\textrm{mult}_x C.\]
	If $C$ is contracted by $\pi$, then we have $-K_S\cdot C=2\mult_x C$ by Lemma \ref{Contraction-Curve}. As a consequence, we obtain $\varepsilon(S,-K_S;s)\geq 2$. Then Theorem \ref{Broustet's theorem} shows $d\geq 4$.
\end{proof}

Now we are in the position to prove the main technique theorem in this paper. It gives the classification of Fano threefolds with $\rho(X)\geq 2$ via Seshadric constant $\varepsilon(X,-K_X;1)$.

\begin{thm}\label{fibration iff Seshadri}
	Let $X$ be a nonsingular Fano threefold with $\rho(X)\geq 2$. 
	\begin{enumerate}
		\item $\varepsilon(X,-K_X;1)=1$ if and only if there is a fibration in del Pezzo surfaces $X\rightarrow\bbP^1$ of degree $1$.
		
		\item $\varepsilon(X,-K_X;1)=4/3$ if and only if there is a fibration in del Pezzo surfaces $X\rightarrow\bbP^1$ of degree $2$.
		
		\item $\varepsilon(X,-K_X;1)=3/2$ if and only if there is a fibration in del Pezzo surfaces $X\rightarrow\bbP^1$ of degree $3$.
		
		\item $\varepsilon(X,-K_X;1)=3$ if and only if $X$ is isomorphic to the blow-up of $\bbP^3$ along a smooth curve $C$ of degree $d$ at most $3$ which is contained in a hyperplane.
		
		\item $\varepsilon(X,-K_X;1)=2$ otherwise.
	\end{enumerate}
\end{thm}

\begin{proof}
	The statement (1) follows directly from Corollary \ref{1-Non-Basepoint-Freeness}. To prove (2), by Corollary \ref{1-Non-Basepoint-Freeness} and Theorem \ref{Morphisms-Splittings}, it is enough to show that if $X$ admits a del Pezzo fibration of degree $2$, then $\varepsilon(X,-K_X;1)=4/3$. By definition, we have $\varepsilon(X,-K_X;1)\leq 4/3$ if $X$ admits a del Pezzo fibration of degree $2$. Then Lemma \ref{Fibration-Different} shows that $\vert-K_X\vert$ is base point free. Therefore, according to Theorem \ref{Morphisms-Splittings}, we have $\varepsilon(X,-K_X;1)\geq 4/3$, and consequently $\varepsilon(X,-K_X;1)=4/3$.
	
	To prove (3), by Corollary \ref{1-Non-Basepoint-Freeness} and Theorem \ref{Morphisms-Splittings}, it suffices to show that if $X$ admits a del Pezzo fibration $f\colon X\rightarrow \bbP^1$ of degree $3$, then it does not admit a del Pezzo fibration of degree $\leq 2$. By Lemma \ref{Fibration-Different}, $\vert -K_X\vert$ is base point free. Thus, by Theorem \ref{Morphisms-Splittings}, it remains to exclude the case in which $X$ admits a del Pezzo fibration of degree $2$. If so, then we have $\varepsilon(X,-K_X;1)=4/3$ by (2). Let $-K_X=D_1+D_2$ be a free splitting of $-K_X$. By Theorem \ref{Morphisms-Splittings}, we may assume that $\vert D_1\vert$ is composed with a pencil of del Pezzo surfaces of degree $2$. Let $S$ be a general fiber of $f$. By assumption, $-K_{S}$ has a free splitting of the following form
	\[-K_{S}=-K_X\vert_{S}=D_1\vert_{S}+D_2\vert_{S}.\]
	The slightly modified argument as in the proof of Lemma \ref{Fibration-Different} shows that we have $\varepsilon(S,-K_{S};1)\geq 2$. This contradicts Theorem \ref{Broustet's theorem} and the fact that $S$ is a del Pezzo surface of degree $3$. Hence, $X$ does not admit a del Pezzo fibration of degree $2$.
	
	To prove (4) and (5), if $\ell_X\geq 3$, \cite[Theorem 1.4]{CasagrandeDruel2015} shows that $X$ is isomorphic to the blow-up of $\bbP^3$ along a smooth plane curve of degree $\leq 3$. On the other hand, we always have $\varepsilon(X,-K_X;1)=3$ for such $X$ (see \cite[Theorem 3]{LiuZhuang2018}). If $\ell_X<3$, then we have $\varepsilon(X,-K_X;1)\leq 2$ and the theorem now follows from Theorem \ref{Morphisms-Splittings}, Corollary \ref{1-Non-Basepoint-Freeness} and (1)-(3).
\end{proof}

\begin{remarks}\label{Special-Spliting-Check}
	\begin{enumerate}
		\item The same argument as in the proof of (3) can be modified to show that the del Pezzo fibration in (2) and (3) is unique.
		
		\item According to Theorem \ref{Morphisms-Splittings} and Theorem \ref{fibration iff Seshadri}, to see if a nonsingular Fano threefold $X$ admits a del Pezzo fibration of degree $2$ or $3$, it suffices to check it for the morphisms induced by an arbitrary free splitting of $-K_X$. 
	\end{enumerate}
\end{remarks}

\section{Del Pezzo fibrations of small degree}

This section is devoted to complete the proof of Theorem \ref{Intro Fano threefolds Picard number large Seshadri}. According to Theorem \ref{fibration iff Seshadri}, we need to find out all Fano threefolds admitting a del Pezzo fibration of degree $2$ or $3$. Nevertheless, as explained in Remark \ref{Special-Spliting-Check}, it suffices to consider the morphisms induced by an arbitrary free splitting of $-K_X$.

\subsection{General results}

Following \cite[Proposition 2.12, 2.13 and 2.14]{MoriMukai1986}, we introduce the following notion for the convenience.
\begin{enumerate}
	\item[--] Let $X$ be a nonsingular Fano threefold. We say that $X$ satisfies ($\clubsuit$) if $X$ can be obtained from blowing-up of a nonsingular threefold $Y$ along a smooth (but possibly disconnected) curve $C$, where $C$ is an intersection of members of a complete linear system $\vert L\vert$.
\end{enumerate}

We start by collecting and reformulating some import properties of the morphisms induced by free splittings of $-K_X$. 

\begin{thm}\cite[\S\,7]{MoriMukai1986}\label{Intersetions-Properties}
	Let $X$ be a nonsingular Fano threefold with $\rho(X)\geq 2$. If 
	\begin{multline*}
	\gimel(X)\not\in\{2.2,2.6,2.8,2.18,2.24,2.32,2.34-2.36,3.1-3.3,3.8,\\
	3.17,3.19,3.27,3.28,3.31,4.1,4.10,5.3,6.1-10.1\},
	\end{multline*}
	then $X$ satisfies ($\clubsuit$). Moreover, let $E$ be the exceptional divisor of the blow-up $f\colon X\rightarrow Y$ and set $D_1=f^*L-E$. Then the following statements hold.
	\begin{enumerate}
		\item There exists a divisor $D_2$ such that $-K_X=D_1+D_2$ is a splitting, which is free if $\vert-K_X\vert$ is base point free. Moreover, if $\vert-K_X\vert$ is base point free, then the linear system $\vert D_2\vert$ is composed with a pencil of surfaces if and only if $\gimel(X)=3.5$.
		
		\item The smooth curve $C$ is a complete intersection of members of $\vert L\vert$ if and only if 
		\begin{multline*}
		\gimel(X)\in\{2.1,2.3-2.5,2.7,2.10,2.14,2.15,2.23,2.25,2.28-2.30,2.33,\\
		3.4,3.7,3.11,3.24,3.26,4.4,4.9,5.1\}.
		\end{multline*}
	\end{enumerate} 
\end{thm}

\begin{proof}
	By definition, if $X$ lies in one of the situations of Proposition 2.12, Proposition 2.13 and Proposition 2.14 in \cite{MoriMukai1986}, then $X$ satisfies $(\clubsuit)$. Moreover, set $D_1=f^*L-E$ and $D_2=-K_X-D_1$. If $\vert-K_X\vert$ is base point free, then $\vert D_2\vert$ is composed with a pencil of surfaces if and only if $X$ satisfies the assumptions of \cite[Proposition 2.14]{MoriMukai1986}. If $X$ is a nonsingular Fano threefold with $\rho(X)=2$, the theorem follows form \cite[(7.1)-(7.3), (7.8)]{MoriMukai1986}. The rest can be checked similarly.
\end{proof}

First we consider the nonsingular Fano threefolds with $\rho=2$. Recall that a nonsingular Fano threefold $X$ is said to be \emph{imprimitive} if it is isomorphic to the blow-up of another nonsingular Fano threefold with center along a smooth irreducible curve. Otherwise $X$ is said to be \emph{primitive}.

\begin{prop}\label{Picard-Number-Two-Cases}
	Let $X$ be a nonsingular Fano threefold with $\rho(X)=2$.
	\begin{enumerate}
		\item There exists a del Pezzo fibration $X\rightarrow \bbP^1$ of degree $2$ if and only if $\gimel(X)\in\{2.2,2.3\}$.
		
		\item There exists a del Pezzo fibration $X\rightarrow \bbP^1$ of degree $3$ if and only if $\gimel(X)\in\{2.4,2.5\}$. 
	\end{enumerate}
\end{prop}

\begin{proof}
	If $\gimel(X)\in\{2.3,2.4,2.5\}$, $X$ is the blow-up of a nonsingular Fano threefold $Y$ with index $r$ along an irreducible smooth curve $C$. Let $H$ be the fundamental divisor of $X$. Then $C$ is a complete intersection of members of $\left\vert (r-1)H\right\vert$. Let $E$ be the exceptional divisor of the blow-up $f\colon X\rightarrow Y$. Then the linear system $\left\vert(r-1)f^*H-E\right\vert$ defines a del Pezzo fibration of degree $2$ (resp. $2$, $3$) if $\gimel(X)=2.3$ (resp. $\gimel(X)=2.4, 2.5$). 
	
	If $\gimel(X)=2.2$, $X$ is a double cover of $\bbP^1\times\bbP^2$ ramified along a divisor of bidegree $(2,4)$. Let $X\rightarrow \bbP^1\times\bbP^2\rightarrow \bbP^1$ be the composite. Then a straightforward computation shows that the general fiber is a del Pezzo surface of degree $2$.
	
	Now we assume that $X$ admits a del Pezzo fibration $f\colon X\rightarrow \bbP^1$ of degree $d\in\{2, 3\}$. As $\rho(X)=2$, $f$ is an extremal contraction. If $X$ is primitive, by the classification given in \cite[Theorem 1.7]{MoriMukai1983}, we have $\gimel(X)=2.2$. 
	
	If $X$ is imprimitive, then $X$ is the blow-up of another nonsingular Fano threefold $Y$ along an irreducible smooth curve $C$. As $\rho(X)=2$, $Y$ is a nonsingular Fano threefold with index $r$ and $\rho(X)=1$. Denote by $\pi\colon X\rightarrow Y$ the blow-up, by $E$ the exceptional divisor of $\pi$ and by $H$ the ample generator of $\pic(Y)$. Note that we have $r\geq 2$ by \cite[Proposition 7.1.5]{Shafarevich1999}. Thanks to \cite[Corollary 7.1.2]{Shafarevich1999}, $C$ should be a scheme-theoretical intersection of divisors from $\left\vert(r-1)H\right\vert$. In particular, the contraction $f$ is induced by $\left\vert(r-1)\pi^*H-E\right\vert$. The linear system $\left\vert(r-1)\pi^*H-E\right\vert$ is composed with a pencil of surfaces if and only if $C$ is a complete intersection of members of $\left\vert(r-1)H\right\vert$ and if so the general fiber of $f$ is a del Pezzo surface of degree $d=(r-1)H^3$. On the other hand, $(r-1)H^3=2$ if and only if $(r,H^3)\in\{(3,1),(2,2)\}$. Nevertheless, if $r=3$, then $Y$ is a quadric threefold and we have $H^3=2$. Therefore the case $(r,H^3)=(3,1)$ does not happen. Hence, if $f$ is a del Pezzo fibration of degree $2$, then $(r,H^3)=(2,2)$, i.e. $\gimel(X)=2.3$. The same argument can be applied to the case $(r-1)H^3=3$ to yield $\gimel(X)\in \{2.4, 2.5\}$.
\end{proof}

Next we consider the nonsingular Fano threefolds satisfying $(\clubsuit)$.

\begin{prop}\label{Intersections-Cases}
	Let $X$ be a nonsingular Fano threefold satisfying ($\clubsuit$). 
	\begin{enumerate}
		\item There exists a del Pezzo fibration $X\rightarrow \bbP^1$ of degree $2$ if and only if $\gimel(X)=2.3$.
		
		\item There exists a del Pezzo fibration $X\rightarrow \bbP^1$ of degree $3$ if and only if $\gimel(X)\in \{2.4, 2.5\}$. 
	\end{enumerate}
\end{prop}

\begin{proof}
	Firstly we consider the case $\gimel(X)=3.5$. By \cite{MoriMukai1981/82}, $X$ is the blow-up of $\bbP^1\times\bbP^2$ along a curve $C$ of bidegree $(5,2)$ such that the composition $C\hookrightarrow \bbP^1\times\bbP^2\rightarrow\bbP^2$ is an embedding. Let $E$ be exceptional divisor of the blow-up $f\colon X\rightarrow Y=\bbP^1\times\bbP^2$, and let $g\colon\bbP^1\times\bbP^2$ be the first projection. Set $L=\cO_Y(-K_Y)\otimes g^*\cO_{\bbP^1}(1)$. Then $C$ is an intersection (not complete intersection) of members of $\vert L\vert$. Moreover, set $D_1=f^*L-E$ and $D_2=f^*g^*\cO_{\bbP^1}(1)$. By \cite[Proposition 2.14]{MoriMukai1986} or \cite[(7.17)]{MoriMukai1986}, $-K_X=D_1+D_2$ is a free splitting. Thus, $\vert D_1\vert$ is not composed with a pencil of surfaces and $\vert D_2\vert$ is composed with a pencil of surfaces. The induced morphism by $\vert D_2\vert$ is exactly $g\circ f\colon X\rightarrow \bbP^1$. Let $S$ be a general fiber of $g\circ f$. From the construction, $S$ is the blow-up of $\bbP^2$ at $5$ points. Therefore, $S$ is a del Pezzo surface of degree $4$.
	
	Secondly, we consider the case in which $\vert D_2\vert$ is not composed with a pencil of surfaces. By Lemma \ref{Criterion-Complete-Intersection}, $\vert D_1\vert$ is composed with a pencil of surfaces if and only if $C$ is a complete intersection. In this case, the general fiber $S$ of the morphism $\Phi_{\vert D_1\vert}\colon X\rightarrow\bbP^1$ is an irreducible del Pezzo surface of degree $D_1\cdot D_2^2$. As the case $\rho(X)=2$ is already considered in Proposition \ref{Picard-Number-Two-Cases}, it remains to consider the case $\rho(X)\geq 3$. Then the quantity $D_1\cdot D_2^2$ is computed in the Appendix and we see that $S$ is always a del Pezzo surface of degree $\geq 4$.
\end{proof}

Finally we consider the nonsingular Fano threefolds of product type.

\begin{lemma}\label{Products-Cases}
	Let $X$ be a nonsingular Fano threefold such that $\rho(X)\geq 6$ or $\gimel(X)\in \{3.27,3.28,4.10,5.3\}$. Then $X$ admits a del Pezzo fibration of degree $d\leq 3$ if and only if $X\simeq\bbP^1\times S$, where $S$ is a del Pezzo surface of degree $d$.
\end{lemma}

\begin{proof}
	One implication is clear. Now we assume that $X$ admits a del Pezzo fibration of degree $d\leq 3$. By the assumption, $X$ is isomorphic to $\bbP^1\times S$, where $S$ is a del Pezzo surface of degree $d$. Set $D_1=p_1^*(-K_{\bbP^1})$ and $D_2=p_2^*(-K_S)$. If $d=1$, then $\vert-K_X\vert$ is not base point free and consequently $S$ is a del Pezzo surface of degree $1$ by Corollary \ref{1-Non-Basepoint-Freeness}. If $d\geq 2$, then $-K_X=D_1+D_2$ is a free splitting, and we can conclude by Theorem \ref{fibration iff Seshadri}.
\end{proof}

\subsection{Remaining cases} 

As a consequence of Theorem \ref{Intersetions-Properties}, Proposition \ref{Picard-Number-Two-Cases}, Proposition \ref{Intersections-Cases} and Lemma \ref{Products-Cases}, it remains to consider the nonsingular Fano threefolds $X$ such that 
\[\gimel(X)\in\{3.1-3.3,3.8,3.17,3.19,3.31, 4.1\}.\]

\begin{longtable}{|M{0.5cm}|M{11cm}|}
	\hline
	$\gimel$ &  brief description \\\hline
	
	$3.1$  &  a double cover of $\bbP^1\times\bbP^1\times\bbP^1$ ramified along a divisor of tridegree $(2,2,2)$ \\\hline
	
	$3.2$  &  a divisor from $\vert \xi^{\otimes 2}\otimes\cO(2,3)\vert$ on  the $\bbP^2$-bundle $\bbP(\cE)$ over $\bbP^1\times\bbP^1$, where $\cE=\cO\oplus \cO(-1,-1)^{\oplus 2}$ and $\xi$ is the tautological bundle $\cO_{\bbP(\cE)}(1)$\\\hline
	
	$3.3$  &   a divisor on $\bbP^1\times\bbP^1\times\bbP^1$ of tridegree $(1,1,2)$\\\hline
	
	$3.8$  & a divisor from $\vert p_1^*g^*\cO(1)\otimes p_2^*\cO(2)\vert$ on $\bbF_1\times \bbP^2$, where $p_i$ is the $i$-th projection and $g\colon \bbF_1\rightarrow \bbP^2$ is the blow-up \\\hline
	
	$3.17$  & a divisor on $\bbP^1\times\bbP^1\times\bbP^2$ of tridegree $(1,1,1)$  \\\hline
	
	$3.19$  & the blow-up of $Q^3\subset\bbP^4$ at two non-collinear points    \\\hline
	
	$3.31$  &  $\bbP(\cO\oplus\cO(1,1))$ over $\bbP^1\times\bbP^1$  \\\hline
	
	$4.1$  &  divisor on $\bbP^1\times\bbP^1\times\bbP^1\times\bbP^1$ of multidegree $(1,1,1,1)$ \\\hline
\end{longtable}

\begin{prop}
	If $\gimel(X)\in \{3.1,3.3,3.17,4.1\}$, then $\varepsilon(X,-K_X;1)=2$.
\end{prop}

\begin{proof}
	In these four cases, it is easily seen that we can write $-K_X=D_1+D_2+D_3$ with all $\vert D_i\vert$ free linear system. Let $x\in X$ be a very general point and let $C$ be an irreducible curve passing through $x$. If there exist $D_i$ and $D_j$ with $i\not=j$ such that $D_i\cdot C>0$ and $D_j\cdot C>0$, then we have $-K_X\cdot C\geq 2\mult_x C$ as before. If there exist $D_i$ and $D_j$ with $i\not=j$ such that $D_i\cdot C=D_j\cdot C=0$, then $C$ is contained in a fiber of the map $\pi=(\pi_i,\pi_j)\colon X\rightarrow Y_i\times Y_j$, where $\pi_i$ and $\pi_j$ are the morphisms induced by $\vert D_i\vert$ and $\vert D_j\vert$, respectively. As $\dim(Y_i\times Y_j)\geq 2$ and $x$ is very general, it follows that $\pi$ is a fibration in curves. By Lemma \ref{Contraction-Curve}, we get $\varepsilon(X,-K_X;1)=2\mult_x C$. Thus we have $\varepsilon(X,-K_X;x)\geq 2$ and we see that the equality holds from Theorem \ref{fibration iff Seshadri}.
\end{proof}

\begin{prop}
	If $\gimel(X)=3.2$, denote by $f$ the composite $X\rightarrow \bbP^1\times \bbP^1\rightarrow \bbP^1$, where the last morphism is the projection to the first factor, then $f$ is a del Pezzo fibration of degree $3$.
\end{prop}

\begin{proof}
	Denote by $H_1$ and $H_2$ the line bundles $\pi^*\cO(1,0)$ and $\pi^*\cO(0,1)$ respectively, where $\pi\colon\bbP(\cE)\rightarrow\bbP^1\times\bbP^1$ is the projection. Then we have $-K_X=(\xi+2H_1+H_2)\vert_X$. Set $D_2=(\xi+H_1+H_2)\vert_X$ and $D_1=H_1\vert_X$. Then $-K_X=D_1+D_2$ is a free splitting and $\vert D_1\vert$ is composed with a rational pencil of surfaces. As $X\sim 2\xi+2H_1+3H_2$, we have
	\begin{align*}
	(K_X+D_1)^2\cdot D_1	& = ((-\xi-H_1-H_2)\vert_X)^2\cdot H_1\vert_X\\
	& = (\xi^2\cdot H_1+2\xi\cdot H_1\cdot H_2)\cdot(2\xi+2H_1+3H_2)\\
	& = 2\xi^3\cdot H_1+5\xi^2\cdot H_1\cdot H_2
	\end{align*}
	By the relation $(\xi+H_1+H_2)^2\cdot\xi=0$ (see \cite[(7.14.2)]{MoriMukai1986}), we obtain 
	\[\xi^3=-2\xi^2\cdot (H_1+H_2)-2\xi\cdot H_1\cdot H_2.\]
	This yields
	\[(K_X+D_1)^2\cdot D_1=3\xi^2\cdot H_1\cdot H_2=3.\]
	Since $\vert-K_X\vert$ is base point free, the general fiber of $f$ is a smooth irreducible del Pezzo surface of degree $3$.
\end{proof}

\begin{prop}
	If $\gimel(X)\in\{3.8,3.19,3.31\}$, then there exists a free splitting $-K_X=D_1+D_2$ such that $\vert D_1\vert$ and $\vert D_2\vert$ are both not composed with a pencil of surfaces. In particular, $X$ does not admit a del Pezzo fibration of degree $\leq 3$.
\end{prop}

\begin{proof}
	\textit{Case 1. $\gimel(X)=3.8$.} Denote by $C$ the exceptional curve of $g$ and set $E=p_1^{-1}(C)$. Let $H_1$ and $H_2$ be the line bundles $p_1^*g^*\cO_{\bbP^2}(1)$ and $p_2\cO_{\bbP^2}(1)$, respectively. Then we have
	\[-K_X=\left(2H_1+H_2-E\right)\vert_X\]
	We set $D_1=(2H_1-E)\vert_X$ and $D_2=H_2\vert_X$. Then $-K_X=D_1+D_2$ is a free splitting such that $\vert D_2\vert$ is not composed with a pencil of surfaces. On the other hand, we have
	\begin{align*}
	D_1^2\cdot (H_1+H_2)\vert_X = 8H_1^2H_2^2+2E^2H_2^2=6>0.
	\end{align*}
	Thus the numerical dimension of $D_1$ is at least $2$. In particular, $\vert D_1\vert$ is not composed with a pencil of surfaces.
	
	\textit{Case 2. $\gimel(X)=3.19$.} Denote by $L$ the line bundle $\cO_{Q^3}(1)$ and by $f\colon X\rightarrow Q^3$ the blow-up. By \cite[(7.11.1)]{MoriMukai1986}, we have a free splitting
	\[-K_X\sim f^*L+2(f^*L-E_1-E_2),\]
	where $E_1$ and $E_2$ are the exceptionl divisors of $f$ over $p$ and $q$, respectively. We set $D_1=f^*L$ and $D_2=2(f^*L-E_1-E_2)$. Then it is easily to see that $\vert D_1\vert$ is not composed with a pencil of surfaces. On the other hand, note that we have
	\[D_2^2\cdot f^*L=4(f^*L)^3=8>0.\]
	Thus the numerical dimension of $D_2$ is at least $2$, and consequently $\vert D_2\vert$ is not composed with a pencil of surfaces.
	
	\textit{Case 3. $\gimel(X)=3.31$.} Denote by $\cE$ the vector bundle $\cO\oplus(1,1)$ over $\bbP^1\times\bbP^1$. Let $\pi\colon \bbP(\cE)\rightarrow \bbP^1\times\bbP^1$ the projection. Set $D_1=\cO_{\bbP(\cE)}(2)$ and $D_2=\pi^*\cO(1,1)$. Then $-K_X\sim D_1+D_2$ is a free splitting. As $\cO(1,1)$ is very ample on $\bbP^1\times\bbP^1$, we see that $\vert D_2\vert$ is not composed with a pencil of surfaces. On the other hand, we have
	\[D_1^2(D_1+3D_2)=(-K_X)^3-3D_1\cdot D_2^2=40>0,\]
	Similarly as above, $\vert D_1\vert$ is not composed with a pencil of surfaces.
\end{proof}

\subsection{Conclusion}

We summarize the main results proved in the previous two subsections in the following theorem. Combining it with Theorem \ref{fibration iff Seshadri} will immediately yield Theorem \ref{Intro Fano threefolds Picard number large Seshadri}.

\begin{thm}
	Let $X$ be a nonsingular Fano threefold.
	\begin{enumerate}
		\item $X$ admits a del Pezzo fibration of degree $2$ if and only if $\gimel(X)\in \{2.2, 2.3, 9.1\}$.
		
		\item $X$ admits a del Pezzo fibration of degree $3$ if and only if $\gimel(X)\in\{2.4, 2.5, 3.2, 8.1\}$.
	\end{enumerate}
\end{thm}

\begin{appendix}
	
\section*{Appendix}
	
	Let $X$ be a nonsingular Fano threefold with $\rho(X)\geq 3$ satisfying $(\clubsuit)$. If $C\subset Y$ is a complete intersection of members from $\vert L\vert$ such that $X$ is isomorphic to the blow-up of $Y$ along $C$ with natural morphism $f\colon X\rightarrow Y$, then $\vert f^*L-E\vert$ is composed with a pencil of del Pezzo surfaces of degree $d=L^2\cdot (-K_Y-L)$, where $E$ is the exceptional divisor of $f\colon X\rightarrow Y$. In the proof of Proposition \ref{Intersections-Cases}, we claim that we have always $d\geq 4$. In the following, we give the details of the calculation. 
	
	\begin{longtable}{|M{0.5cm}|M{5cm}|M{5cm}|M{0.5cm}|}
		\hline
		$\gimel$ &  $Y$    &   $L$  &  $d$  \\\hline
		
		$3.4$  &  $f\colon Y\rightarrow \bbP^1\times\bbP^2$ is a double cover whose branch locus is a divisor of bidegree $(2,2)$. &   $f^*p_2^*\cO_{\bbP^2}(1)$, where $p_2\colon \bbP^1\times\bbP^2\rightarrow\bbP^2$ is the projection to the second factor  &  $4$ \\\hline
		
		$3.7$  &  $W\subset \bbP^2\times\bbP^2$ a smooth divisor of bidegree $(1,1)$    &   $-\frac{1}{2}K_W$  &  $6$ \\\hline
		
		$3.11$  &  $\pi\colon V_7\rightarrow \bbP^3$ is the blow-up of $\bbP^3$ at a point $p$ with the exceptional divisor $E$  &   $-\frac{1}{2}K_{V_7}$  &  $7$ \\\hline
		
		$3.24$  &  $W\subset \bbP^2\times\bbP^2$ a smooth divisor of bidegree $(1,1)$    &   $\cO_W\otimes \cO_{\bbP^2\times\bbP^2}(0,1)$  &  $8$ \\\hline
		
		$3.26$  &  $\pi\colon V_7\rightarrow \bbP^3$ is the blow-up of $\bbP^3$ at a point $p$ with the exceptional divisor $E$    &   $\pi^*\cO_{\bbP^3}(1)$  &  $9$ \\\hline
		
		$4.4$  &  $\pi\colon Y\rightarrow Q^3$ is the blow-up of $Q^3\subset\bbP^4$ with center two points $x_1$ and $x_2$ on it which are not colinear  with exceptional divisors $E_1$ and $E_2$.      &  $\pi^*\cO_{Q^3}(1)\otimes\cO_Y(-E_1-E_2)$   &  6   \\\hline
		
		$4.9$  &  $f\colon Y\rightarrow\bbP^3$ is obtained by first blowing up along a line $\ell$ and then blowing-up an exceptional line of the first blowing-up &   $f^*\cO_{\bbP^3}(1)$  &  $8$  \\\hline
		
		$5.1$  &  $\pi\colon Y\rightarrow Q^3$ is the blow-up of $Q^3\subset\bbP^4$ three points $x_i$ on a conic on it with exceptional divisors $E_i$ ($1\leq i\leq 3$).    &   $\pi^*\cO_{Q^3}(1)\otimes\cO_Y(-E_1-E_2-E_3)$  &  $5$  \\\hline
	\end{longtable}
	
	\begin{enumerate}
		\item $\gimel(X)=3.4$. Denote by $H$ the line bundle $f^*\cO_{\bbP^1\times\bbP^2}(1,0)$. By ramification formula, we have $K_Y=-H-2L$. It follows
		\[(-K_Y-L)^2\cdot L=(H+L)^2\cdot L=2H\cdot L^2=4.\]
		
		\item $\gimel(X)=3.7$. Here $Y$ is the del Pezzo threefold $W$ of degree $6$. As $L=-1/2K_W$, we obtain
		\[(-K_Y-L)^2\cdot L=-\frac{1}{8}K_Y^3=\frac{1}{8}(-K_W)^3=6.\]
		
		\item $\gimel(X)=3.11$. Here $Y$ is a del Pezzo threefold $V_7$ of degree $7$. As $L=-1/2 K_{Y}$, we have
		\[(-K_{Y}-L)^2\cdot L=-\frac{1}{8}K_Y^3=\frac{1}{8}(-K_{V_7})^3=7.\]
		
		\item $\gimel(X)=3.24$. Denote the line bundles $\cO_{\bbP^2\times\bbP^2}(1,0)$ and $\cO_{\bbP^2\times\bbP^2}(0,1)$ by $H_1$ and $H_2$, respectively. Then $W=H_1+H_2$ and $K_W=(-2H_1-2H_2)\vert_W$ and we have
		\[(-K_W-L)^\cdot L=(2H_1\vert_W+H_2\vert_W)^2\cdot H_2\vert_W=8H_1^2\cdot H_2^2=8.\]
		
		\item$\gimel(X)=3.26$. Here $Y$ is a del Pezzo threefold $V_7$ of degree $7$. As $-K_{V_7}=-4L+2E$, we obtain
		\[(-K_{V_7}-L)^2\cdot L=(3L-2E)^2\cdot L=9L^3=9.\]
		
		\item $\gimel(X)=4.4$. Denote by $H$ the line bundle $\pi^*\cO_{Q^3}(1)$. Then $L=H-E_1-E_2$ and $K_Y=-3H+2E_1+2E_2$. It follows
		\[(-K_Y-L)^2\cdot L=(2H-E_1-E_2)^2\cdot (H-E_1-E_2)=6.\]
		
		\item $\gimel(X)=4.9$. Denote by $E_1$ the strict transform of the exceptional divisor of the first blowi-up and denote by $E_2$ the exceptional divisor of the second blow-up. Then we have $K_Y=-4L+E_1+2E_2$. It follows
		\[(-K_Y-L)^2\cdot L=(3L-E_1-2E_2)^2\cdot L=9L^3+L\cdot E_1^2=8.\]
		
		\item $\gimel(X)=5.1$. Let $H$ be the line bundle $\pi^*\cO_{Q^3}(1)$. Then $L=H-E_1-E_2-E_3$ and $K_Y=-3H+2E_1+2E_2+2E_3$. It follows
		\[(-K_Y-L)^2\cdot L=(2H-E_1-E_2-E_3)^2\cdot(H-E_1-E_2-E_3)=5.\]
	\end{enumerate}
	
\end{appendix}

\def\cprime{$'$}

\renewcommand\refname{Reference}
\bibliographystyle{alpha}
\bibliography{seshadriconstants}

\begin{thebibliography}{CMSB02}

\bibitem[Bro06]{Broustet2006}
Ama\"el Broustet.
\newblock Constantes de {S}eshadri du diviseur anticanonique des surfaces de
  del {P}ezzo.
\newblock {\em Enseign. Math. (2)}, 52(3-4):231--238, 2006.

\bibitem[Bro09]{Broustet2009}
Ama\"el Broustet.
\newblock Non-annulation effective et positivit\'e locale des fibr\'es en
  droites amples adjoints.
\newblock {\em Math. Ann.}, 343(4):727--755, 2009.

\bibitem[CD15]{CasagrandeDruel2015}
Cinzia Casagrande and St\'ephane Druel.
\newblock Locally unsplit families of rational curves of large anticanonical
  degree on {F}ano manifolds.
\newblock {\em Int. Math. Res. Not. IMRN}, (21):10756--10800, 2015.

\bibitem[Cha10]{Chan2010}
Kungho Chan.
\newblock A lower bound on {S}eshadri constants of hyperplane bundles on
  threefolds.
\newblock {\em Math. Z.}, 264(3):497--505, 2010.

\bibitem[CMSB02]{ChoMiyaokaShepherd-Barron2002}
Koji Cho, Yoichi Miyaoka, and Nicholas~I. Shepherd-Barron.
\newblock Characterizations of projective space and applications to complex
  symplectic manifolds.
\newblock In {\em Higher dimensional birational geometry ({K}yoto, 1997)},
  volume~35 of {\em Adv. Stud. Pure Math.}, pages 1--88. Math. Soc. Japan,
  Tokyo, 2002.

\bibitem[CN14]{CasciniNakamaye2014}
Paolo Cascini and Michael Nakamaye.
\newblock Seshadri constants on smooth threefolds.
\newblock {\em Adv. Geom.}, 14(1):59--79, 2014.

\bibitem[Dem92]{Demailly1992}
Jean-Pierre Demailly.
\newblock Singular {H}ermitian metrics on positive line bundles.
\newblock In {\em Complex algebraic varieties ({B}ayreuth, 1990)}, volume 1507
  of {\em Lecture Notes in Math.}, pages 87--104. Springer, Berlin, 1992.

\bibitem[DH17]{DedieuHoering2017}
Thomas Dedieu and Andreas H\"oring.
\newblock Numerical characterisation of quadrics.
\newblock {\em Algebr. Geom.}, 4(1):120--135, 2017.

\bibitem[EKL95]{EinKuechleLazarsfeld1995}
Lawrence Ein, Oliver K\"uchle, and Robert Lazarsfeld.
\newblock Local positivity of ample line bundles.
\newblock {\em J. Differential Geom.}, 42(2):193--219, 1995.

\bibitem[EL93]{EinLazarsfeld1993}
Lawrence Ein and Robert Lazarsfeld.
\newblock Seshadri constants on smooth surfaces.
\newblock {\em Ast\'erisque}, (218):177--186, 1993.
\newblock Journ\'ees de G\'eom\'etrie Alg\'ebrique d'Orsay (Orsay, 1992).

\bibitem[Flo13]{Floris2013}
Enrica Floris.
\newblock Fundamental divisors on {F}ano varieties of index {$n-3$}.
\newblock {\em Geom. Dedicata}, 162:1--7, 2013.

\bibitem[ILP13]{IltenLewisPrzyjalkowski2013}
Nathan~Owen Ilten, Jacob Lewis, and Victor Przyjalkowski.
\newblock Toric degenerations of {F}ano threefolds giving weak
  {L}andau-{G}inzburg models.
\newblock {\em J. Algebra}, 374:104--121, 2013.

\bibitem[IM14]{ItoMiura2014}
Atsushi Ito and Makoto Miura.
\newblock Seshadri constants and degrees of defining polynomials.
\newblock {\em Math. Ann.}, 358(1-2):465--476, 2014.

\bibitem[IP99]{Shafarevich1999}
Vasily~Alexeevich Iskovskikh and Yuri~G. Prokhorov.
\newblock {\em Algebraic geometry. {V}}, volume~47 of {\em Encyclopaedia of
  Mathematical Sciences}.
\newblock Springer-Verlag, Berlin, 1999.

\bibitem[Ito14]{Ito2014}
Atsushi Ito.
\newblock Seshadri constants via toric degenerations.
\newblock {\em J. Reine Angew. Math.}, 695:151--174, 2014.

\bibitem[Kol96]{Kollar1996}
J{\'a}nos Koll{\'a}r.
\newblock {\em Rational curves on algebraic varieties}, volume~32 of {\em
  Ergebnisse der Mathematik und ihrer Grenzgebiete. 3. Folge. A Series of
  Modern Surveys in Mathematics}.
\newblock Springer-Verlag, Berlin, 1996.

\bibitem[Laz04]{Lazarsfeld2004}
Robert Lazarsfeld.
\newblock {\em Positivity in algebraic geometry. {I}}, volume~48 of {\em
  Ergebnisse der Mathematik und ihrer Grenzgebiete. 3. Folge. A Series of
  Modern Surveys in Mathematics}.
\newblock Springer-Verlag, Berlin, 2004.
\newblock Classical setting: line bundles and linear series.

\bibitem[Lee03]{Lee2003}
Seunghun Lee.
\newblock Seshadri constants and {F}ano manifolds.
\newblock {\em Math. Z.}, 245(4):645--656, 2003.

\bibitem[Lee04]{Lee2004a}
Seunghun Lee.
\newblock Linear systems on {F}ano threefolds. {I}.
\newblock {\em Comm. Algebra}, 32(7):2711--2721, 2004.

\bibitem[Liu17]{Liu2017}
Jie Liu.
\newblock Second chern class of {F}ano manifolds and anti-canonical geometry.
\newblock {\em Math. Ann.}, to appear, 2017.
\newblock https://doi.org/10.1007/s00208-018-1702-z.

\bibitem[LZ18]{LiuZhuang2018}
Yuchen Liu and Ziquan Zhuang.
\newblock Characterization of projective spaces by {S}eshadri constants.
\newblock {\em Math. Z.}, 289(1-2):25--38, 2018.

\bibitem[Miy04]{Miyaoka2004}
Yoichi Miyaoka.
\newblock Numerical characterisations of hyperquadrics.
\newblock In {\em Complex analysis in several variables---{M}emorial
  {C}onference of {K}iyoshi {O}ka's {C}entennial {B}irthday}, volume~42 of {\em
  Adv. Stud. Pure Math.}, pages 209--235. Math. Soc. Japan, Tokyo, 2004.

\bibitem[MM81]{MoriMukai1981/82}
Shigefumi Mori and Shigeru Mukai.
\newblock Classification of {F}ano {$3$}-folds with {$B_{2}\geq 2$}.
\newblock {\em Manuscripta Math.}, 36(2):147--162, 1981.

\bibitem[MM83]{MoriMukai1983}
Shigefumi Mori and Shigeru Mukai.
\newblock On {F}ano {$3$}-folds with {$B_{2}\geq 2$}.
\newblock In {\em Algebraic varieties and analytic varieties ({T}okyo, 1981)},
  volume~1 of {\em Adv. Stud. Pure Math.}, pages 101--129. North-Holland,
  Amsterdam, 1983.

\bibitem[MM86]{MoriMukai1986}
Shigefumi Mori and Shigeru Mukai.
\newblock Classification of {F}ano {$3$}-folds with {$B_2\geq 2$}. {I}.
\newblock In {\em Algebraic and topological theories ({K}inosaki, 1984)}, pages
  496--545. Kinokuniya, Tokyo, 1986.

\bibitem[MM03]{MoriMukai2003}
Shigefumi Mori and Shigeru Mukai.
\newblock Erratum: ``{C}lassification of {F}ano 3-folds with {$B_2\geq 2$}''
  [{M}anuscripta {M}ath. {\bf 36} (1981/82), no. 2, 147--162; {MR}0641971
  (83f:14032)].
\newblock {\em Manuscripta Math.}, 110(3):407, 2003.

\bibitem[Nak05]{Nakamaye2005}
Michael Nakamaye.
\newblock Seshadri constants at very general points.
\newblock {\em Trans. Amer. Math. Soc.}, 357(8):3285--3297, 2005.

\end{thebibliography}

\end{document}